\documentclass[11pt]{amsart}

\usepackage{geometry}
\usepackage{graphicx}
\usepackage{amssymb}
\usepackage{epstopdf}
\usepackage{enumerate}
\numberwithin{equation}{section}
\newtheorem{thm}{Theorem}[section]
\newtheorem{lemma}[thm]{Lemma}
\newtheorem{prop}[thm]{Proposition}
\newtheorem{cor}[thm]{Corollary}
\newtheorem*{ack}{Acknowledgement}
\newcommand{\Q}{\mathbb{Q}}
\newcommand{\R}{\mathbb{R}}
\newcommand{\Z}{\mathbb{Z}}

\newcommand{\N}{\mathbb{N}}
\newcommand{\p}{\text{proj}}
\begin{document}

\title[Projections of self-affine carpets]{The Hausdorff dimension of the projections of self-affine carpets}

\author{Andrew Ferguson}
\address{Andrew Ferguson\\Mathematics Institute\\Zeeman Building\\
University of Warwick\\
Coventry\\CV4 7AL\\UK.}
\email{a.j.ferguson@warwick.ac.uk}

\author{Thomas Jordan}
\address{Thomas Jordan\\Department of Mathematics\\University of Bristol\\University Walk\\Clifton\\Bristol BS8 1TW\\UK }
\email{Thomas.Jordan@bris.ac.uk}

\author{Pablo Shmerkin}
\address{Pablo Shmerkin\\School of Mathematics and Centre for Interdisciplinary Computational and Dynamical Analysis, Alan Turing Building \\
University of Manchester, Oxford Road\\
Manchester M13 9PL\\ UK.}
\email{Pablo.Shmerkin@manchester.ac.uk}

\thanks{P.S. acknowledges support from EPSRC grant EP/E050441/1 and the University of Manchester.}


\date{\today}

\maketitle

\begin{abstract}
{We study the orthogonal projections of a large class of self-affine carpets, which contains the carpets of Bedford and McMullen as special cases. Our main result is that if $\Lambda$ is such a carpet, and certain natural irrationality conditions hold, then every orthogonal projection of $\Lambda$ in a non-principal direction has Hausdorff dimension $\min(\gamma,1)$, where $\gamma$ is the Hausdorff dimension of $\Lambda$. This generalizes a recent result of Peres and Shmerkin on sums of Cantor sets.
	
}
\end{abstract}

\section{Introduction and statement of results}

A basic problem in fractal geometry is to understand how the Hausdorff dimension of a set behaves under orthogonal projections. Results which are valid for almost every projection are well-known, going back to the celebrated Marstrand Projection Theorem \cite{Marstrand1954}: let $\Lambda\subset\R^2$ be a Borel set, and denote the orthogonal projection onto a line making angle $\theta$ with the origin by $\p_\theta$. Then
\[
\dim_H(\p_\theta(\Lambda)) = \min(\dim_H(\Lambda),1) \qquad \text{ for almost every } \theta,
\]
where $\dim_H$ stands for Hausdorff dimension. See \cite[Chapter 9]{Mattila1995}, \cite{PeresSchlag2000} and references therein for many extensions of Marstrand's Theorem. We underline that in all cases, the proof is non-constructive and gives no indication of what the exceptional set of directions may look like (other than giving a bound on its dimension). This motivates the following general question: if $\Lambda\subset\mathbb{R}^2$ is dynamically defined, is it possible to determine the set of exceptional directions in Marstrand's Theorem explicitly?

Recall that the one-dimensional Sierpi\'{n}ski gasket $\mathcal{S}$ is defined as
\[
\mathcal{S} = \left\{\sum_{i=1}^\infty 3^{-i} a_i : a_i \in \{ (0,0),(1,0),(0,1)\}\right\}.
\]
A question, attributed to Furstenberg, is whether $\p_\theta(\mathcal{S})$ has Hausdorff dimension one for all $\theta$ with irrational slope; this problem remains open. On the other hand, Moreira (see \cite{Moreira1998}) succeeded in showing that for the cartesian product of certain non-linear dynamically defined Cantor sets on the real line, there are no exceptional directions, other than $\theta=0, \pi/2$ which are trivial ones (his motivation was to find the dimension of the arithmetic sum of two such Cantor sets; note that the arithmetic sum $A+B$ is affinely equivalent to a projection $\p_{\pi/4}(A\times B)$). Peres and Shmerkin \cite{PeresShmerkin2009} obtained an analogous result for products of linear self-similar sets and for self-similar sets in the plane. The main motivating class of examples in their work are the product sets $\Lambda = C_a \times C_b$, where $C_s$ is the central Cantor set which is obtained by replacing the unit interval $[0,1]$ by the union $[!
 0,s]\cup [1-s,1]$ and iterating. All these positive results require an appropriate irrationality assumption; for $\Lambda=C_a\times C_b$ this reduces to $\log b/\log a$ being an irrational number. Very recently, Hochman and Shmerkin \cite{HochmanShmerkin2009} introduced a general approach that unifies and extends all these results.

In the present work we continue this line of research. We focus on a family of dynamically defined fractals generally known as \textbf{self-affine carpets}. These carpets are defined by replacing the unit square $Q$ with a union of pairwise non-overlapping rectangles $\{ S_i(Q)\}$ satisfying some geometric arrangement, and iterating inside each $S_i(Q)$; here the $S_i$ are affine maps with a diagonal linear part. A simple model of self-affine carpets was introduced by Bedford \cite{Bedford1984} and McMullen\cite{McMullen1984}, who independently found a formula for their Hausdorff dimension. Figure \ref{fig:mcmullen} shows a typical Bedford-McMullen carpet. We will be concerned with two more general classes of self-affine carpets, the class studied by Gatzouras and Lalley \cite{GatzourasLalley1992}, and the class recently introduced by Bara\'{n}ski \cite{Baranski2007}; precise definitions are given below. Our main result is that if $\Lambda$ is either a Gatzouras-Lalley or a Bara\'{n}ski carpet, and a natural irrationality condition holds, then
\[
\dim_H(\p_\theta(\Lambda)) = \min(\dim_H(\Lambda),1) \qquad \text{ for all } \theta\in (0,\pi)\setminus\{\pi/2\}.
\]
(See Theorem \ref{main} below for the precise statement). In other words, for these carpets the only possible exceptional directions in Marstrand's Theorem are $0$ and $\pi/2$; it will be clear from the definitions that these directions can indeed be exceptional. We underline that the aforementioned results in \cite{Moreira1998}, \cite{PeresShmerkin2009}, \cite{HochmanShmerkin2009} all rely on the conformality of the underlying constructions. The self-affine carpets we study are intrinsically non-conformal.
\begin{figure}
\begin{center}
\includegraphics[width=0.8\textwidth]{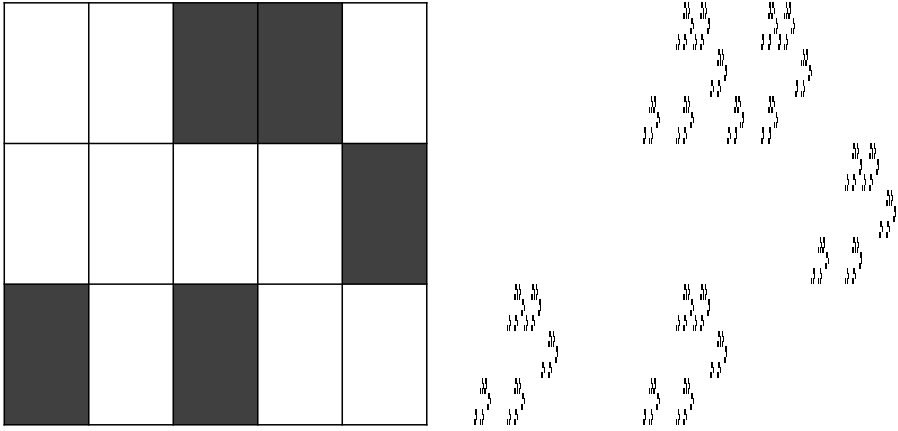}
\end{center}
\caption{A generating pattern of a Bedford-McMullen carpet (left) and the associated invariant set $\Lambda$ (right).}\label{fig:mcmullen}
\end{figure}

Recall that an \textbf{iterated function system} on $\mathbb{R}^d$ is a finite collection of maps $\{S_1,\ldots, S_k\}$ from $\R^d$ into itself which are strictly contractive, i.e. there exists $0<c<1$ such that $|S_i(x)-S_i(y)|<c|x-y|$ for all $x,y\in\mathbb{R}^d$ and all $i$. It is well-known that there exists a unique non-empty compact set $\Lambda$, called the \textbf{attractor} or \textbf{invariant set} of the iterated function system, such that $\Lambda = \cup_i S_i(\Lambda)$. In the case in which all the $S_i$ are homotheties, $\Lambda$ is called a \textbf{self-similar set}. We will be concerned with the case in which all the $S_i$ are affine maps, in which case $\Lambda$ is known as a \textbf{self-affine set}. The reader is referred to \cite{Falconer2003} for further background on iterated function systems and self-similar sets.

One reason why self-affine carpets are of interest is that they are among the simplest planar constructions which are not self-similar. They exhibit phenomena not present in the self-similar setting, like the non-coincidence of Hausdorff and packing dimension (see \cite{GatzourasLalley1992}) or the infinitude of Hausdorff measure in the critical dimension \cite{Peres1994}. Our results suggest that from the point of view of orthogonal projections, these carpets behave rather like products of self-similar sets; compare Theorem 1.1 with \cite[Theorem 5]{PeresShmerkin2009}.

Let us give the precise definitions of the constructions we will consider. The first type was introduced by Gatzouras and Lalley in \cite{GatzourasLalley1992}. For this construction we fix positive integers $m, n_1, n_2, \ldots, n_m$ and set $D=\{(i,j)\,\,:\,\,1\leq i \leq m,\,\, 1\leq j \leq n_i\}$.  Let $\Lambda$ be the unique non-empty compact set satisfying
\begin{equation}\nonumber \Lambda = \bigcup_{(i,j)\in D} S_{i j} \Lambda, \end{equation}
where the map $S_{i j}$ is of the form
\begin{equation}\nonumber S_{i j}(x,y)=(a_{i j}x+c_{i j},b_i y+d_i).\end{equation}
We impose the following conditions on the maps $S_{i j}$: $0<a_{i j}< b_i<1$ for each pair $(i,j)$, and the intervals $\{b_i I + d_i\}_{i=1}^m$, $\{a_{i j} I + c_{i j}\}_{j=1}^{n_i}$ have disjoint interiors, where $I=[0,1]$ is the unit interval; see Figure \ref{fig:glb}. For any $n\in\N$, we will let
\[
\mathbb{P}^n=\left\{(p_1,\ldots,p_n):p_i\geq 0\text{ and }\sum_{i=1}^np_i=1\right\}
\]
be the space of all probability vectors with $n$ elements.
It is proven in \cite{GatzourasLalley1992} (albeit with a slightly different formulation) that the Hausdorff dimension of $\Lambda$ is given by
\begin{equation}
\label{dimH} \dim_H \Lambda =\sup_{\textbf{p}\in \mathbb{P}^m}\left\{\frac{\sum_{i=1}^m p_i \log p_i}{\sum_{i=1}^m p_i\log b_i}+t(\textbf{p})\right\},
\end{equation}
 where $t(\textbf{p})$ is the unique real number satisfying
\begin{equation}
\nonumber\sum_{i=1}^m p_i \log \left(\sum_{j=1}^{n_i} a_{i j}^{t(\textbf{p})}\right)=0.
\end{equation}
We will say a Gatzouras-Lalley construction is of \textbf{irrational type} if either
\begin{enumerate}
\item\label{it1}
There exists $(i,j)\in D$ with $\frac{\log a_{ij}}{\log b_i}\notin\Q$.
\item\label{it2}
There exist $i,j,k$ with $\frac{\log a_{ij}}{\log b_k}\notin\Q$ and $\frac{\log a_{ij}}{\log b_i}$ is not constant for all $i,j$.
\end{enumerate}
A special case of the Gatzouras -Lalley construction is the situation where $n_1=n_2=\cdots=n_m$ and there are constants $0<a<b<1$ such that each $a_{ij}=a$ and $b_i=b$. We will refer to this as the \textbf{homogeneous uniform fibre} case.

\begin{figure}
\begin{center}
\includegraphics[width=0.8\textwidth]{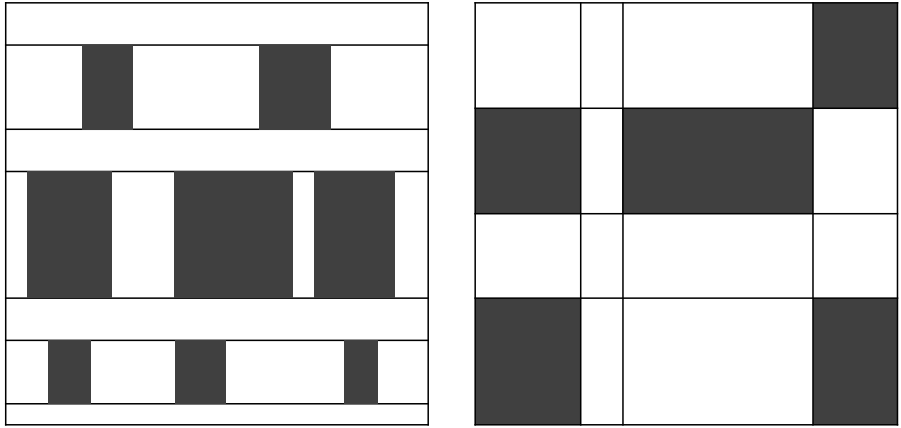}
\caption{Generating patterns for a Gatzouras-Lalley carpet (left) and a Bara\'{n}ski carpet (right).} \label{fig:glb}
\end{center}
\end{figure}

An alternative construction was considered by Bara\'{n}ski in \cite{Baranski2007}. For this construction fix positive integers $m,n$ and let
$D\subseteq\{(i,j):1\leq i\leq n\text{ and }1\leq j\leq m\}$. For each $1\leq i\leq n$ and $1\leq j\leq m$ we fix values $0<a_i, b_j<1$ such that $\sum_{i=1}^n a_i=\sum_{j=1}^m b_j=1$. We then consider the affine iterated function system defined by the maps
\[
 S_{i,j}(x,y)=\left(a_ix+\sum_{l=1}^{i-1}a_l,b_j y+\sum_{l=1}^{j-1}b_l\right)
\]
for each $(i,j)\in D$. Again we let $\Lambda$ be the unique non-empty compact set satisfying
\begin{equation}\nonumber \Lambda = \bigcup_{(i,j)\in D} S_{i j} \Lambda. \end{equation}
(See Figure \ref{fig:glb}). We remark that Bedford-McMullen carpets are the special case of the Bara\'{n}ski construction in which $a_i=1/n$ and $b_j=1/m$ for all $i, j$.

In \cite{Baranski2007} Bara\'{n}ski computed the dimension of such attractors. To state his result we denote the elements of a probability vector $\textbf{p}\in\mathbb{P}^{|D|}$ by $p_{ij}$ where $|D|$ denotes the cardinality of $D$. For $1\leq i\leq n$ let
$q_i(\textbf{p})=\sum_{j:(i,j)\in D} p_{ij}$ and for $1\leq j\leq m$ let $r_j(\textbf{p})=\sum_{i:(i,j)\in D} p_{ij}$. We now define
$$d_x=\sup_{\textbf{p}\in\mathbb{P}^{|D|}}\left\{\frac{\sum_{i=1}^n q_i(\textbf{p})\log q_i(\textbf{p})}{\sum_{i=1}^nq_i(\textbf{p})\log a_i}+\frac{\sum_{ij} p_{ij}\log\left(\frac{p_{ij}}{q_i(\textbf{p})}\right)}{\sum_{j=1}^m r_j(\textbf{p})\log b_j}\right\},$$
and
$$d_y=\sup_{\textbf{p}\in\mathbb{P}^{|D|}}\left\{\frac{\sum_{j=1}^m r_j(\textbf{p})\log r_j(\textbf{p})}{\sum_{j=1}^mr_j(\textbf{p})\log b_j}+\frac{\sum_{ij} p_{ij}\log\left(\frac{p_{ij}}{r_j(\textbf{p})}\right)}{\sum_{i=1}^n q_i(\textbf{p})\log a_i}\right\}.$$
Bara\'{n}ski showed that
$$\dim_H\Lambda=\max\{d_x,d_y\}.$$

We will call the system of \textbf{irrational type} if either:
\begin{enumerate}
\item  We can find $(i,j)\in D$ such that $\frac{\log a_i}{\log b_j}\notin\Q$,
\item  We can find $(i,j)\notin D$ such that $\frac{\log a_i}{\log b_j}\notin\Q$ and there exists $(i,j),(k,l)\in D$ with $\frac{\log a_i}{\log b_j}\neq \frac{\log a_k}{\log b_l}$.
\end{enumerate}

We now state our main result:
\begin{thm} \label{main}
Let $\Lambda$ be the attractor for an iterated function system defined by the Gatzouras-Lalley construction or the Bara\'{n}ski construction. If the system is of irrational type, then
\begin{equation}\nonumber\dim_H\left(\text{proj}_\theta \Lambda \right)=\min(\dim_H\left(\Lambda\right),1)\end{equation}
for all $ \theta\in (0,\pi)\setminus\{\pi/2\}$.\label{thm1}\end{thm}

The Bara\'{n}ski class contains the products $C_a\times C_b$ as special cases, and also the cartesian products of more general self-similar sets in the line; hence, Theorem \ref{main} generalizes \cite[Theorem 2]{PeresShmerkin2009}. We underline that the proofs in \cite{PeresShmerkin2009} rely crucially on the product structure of $\Lambda=C_a\times C_b$. More precisely, in that paper the following fact is key: $\Lambda$ can be written as a union $\cup_i \Lambda_i$, where each $\Lambda_i$ is the image of $\Lambda$ under an affine map of small distortion. This structure breaks down already for Bedford-McMullen carpets such as the one depicted in Figure \ref{fig:mcmullen}. Thus, although we follow the general pattern of proof of \cite[Theorem 2]{PeresShmerkin2009}, we must also introduce new ideas to deal with the lack of product structure.

A key to the proof of Theorem \ref{main} is to first prove the homogeneous uniform fibre Gatzouras-Lalley case, and then use probabilistic ideas to deduce the general case.
\begin{prop}\label{keyprop}
Let $\Lambda$ be an attractor for a homogeneous uniform fibre Gatzouras-Lalley construction. If $\frac{\log a}{\log b}\notin\mathbb{Q}$ then \begin{equation}\nonumber\dim_H\left(\text{proj}_\theta \Lambda \right)=\min(\dim_H\left(\Lambda\right),1)\end{equation}
for all $ \theta\in (0,\pi)\setminus\{\pi/2\}.$
\end{prop}

In \cite{KenyonPeres1996}, Kenyon and Peres studied a generalization of Bedford-McMullen carpets which are modeled by sofic symbolic systems (Recall that a sofic system is a factor of a subshift of finite type on a finite alphabet). We can describe their construction as follows: fix integers $2\le m<n$. Let $X$ be a closed subset of the symbolic space $D^\mathbb{N}$, where $D=\{1,\ldots,n\}\times \{1,\ldots,m\}$, and assume $X$ is $T$-invariant and sofic, where $T$ is the shift map on $D^\mathbb{N}$. For $(i,j)\in D$, write
\[
S_{(i,j)}(x,y) = ((x+i)/n,(y+j)/m).
\]
Further, set
\begin{equation} \label{soficinvariant}
\Lambda = \Lambda(X) = \bigcap_{\ell=1}^\infty \bigcup_{ |\sigma|=\ell, \sigma\in X^*} S_{\sigma_1}\circ\cdots \circ S_{\sigma_\ell}([0,1]\times[0,1]),
\end{equation}
where $X^*$ is the set of all finite strings which can be continued to an infinite string in $X$. It is easy to see that in the case $X=(D')^\mathbb{N}$, where $D'\subset D$, the set $\Lambda$ is a Bedford-McMullen carpet; and thus a Bara\'{n}ski construction, and is of irrational type whenever $\log m/\log n\notin\mathbb{Q}$.

Kenyon and Peres obtained a formula for the Hausdorff dimension of such sofic affine-invariant sets in terms of some random matrix products \cite[Theorems 1.1 and 3.2]{KenyonPeres1996}. Of importance to us is that in the course of the proof of Theorem 3.2, Kenyon and Peres construct Bedford-McMullen carpets $\Lambda_j$ corresponding to a subset $D_j \subset \{0,\ldots, n^j\}\times \{0,\ldots, m^j\}$, such that $\Lambda_j\subset \Lambda$ for all $j$, and $\dim_H(\Lambda_j)\rightarrow \dim_H(\Lambda)$ as $j\rightarrow\infty$. Hence we have the following immediate corollary of Theorem \ref{main}:

\begin{cor}
Suppose that $\log m/\log n$ is irrational and $\Lambda$ is defined by \eqref{soficinvariant} for some sofic system $X$. Then
\[
\dim_H(\p_\theta\Lambda) = \min(\dim_H(\Lambda),1)
\]
for all $ \theta\in (0,\pi)\setminus\{\pi/2\}$.
\end{cor}

When seen as a subset of the torus $T^2$, the set $\Lambda$ is invariant under the expanding toral endomorphism $A(x,y) = (mx, ny) \,(\text{mod}\,1)$. We do not know if the above corollary remains true if $\Lambda$ is an arbitrary closed $A$-invariant subset of the torus of (this would correspond to $X$ being an arbitrary subshift, not necessarily sofic).

Section \ref{sec:dpthm} contains a discrete projection theorem, while Section \ref{sec:keyprop} contains the proof of Proposition \ref{keyprop}. In Section \ref{sec:main} we then use probabilistic ideas and Proposition \ref{keyprop} to complete the proof of Theorem \ref{thm1}. Throughout the article, $\mathcal{L}$ will be used to denote Lebesgue measure.

\section{A discrete projection theorem}
\label{sec:dpthm}

The main idea behind the proof of Proposition \ref{keyprop} is that for a self-affine carpet $\Lambda$, and a fixed angle $\theta\in(0,\pi/2)$, we can construct a subset $\Lambda_\theta\subset \Lambda$, with the following properties:  $\Lambda_\theta$ may be written as an intersection over a nested family of compact sets, i.e. $\Lambda_\theta=\cap_k \cup_{Q\in\mathcal{Q}_k} Q$, such that for each $k$, the elements of the family $\{\text{proj}_\theta (Q)\,\,:\,\,Q\in\mathcal{Q}_k\}$ are suitably separated.  The following theorem gives sufficient conditions for the cardinality of the sets $\mathcal{Q}_k$ being large enough so that the resulting set $\Lambda_\theta$ has dimension that is arbitrarily close to that of $\Lambda$.  It plays a similar role to that of \cite[Proposition 7]{PeresShmerkin2009}, and is a variation on this result.

\begin{thm}Given constants $A>1, A_1,A_2>0$ and $\gamma\in (0,1)$, there exists a constant $\delta$ such that the following holds: Fix $\rho>0$.  Let $\mathcal{Q}$ be a collection of convex subsets of the unit disk such that each element contains a disk of radius $A^{-1}\rho$ and is contained in a disk of radius $A\rho$, and $Q,Q^{\prime}$ have disjoint interiors for any $Q,Q^{\prime}\in\mathcal{Q}$ with $Q\neq Q^{\prime}$. Furthermore suppose that $\mathcal{Q}$ has cardinality at least $\rho^{-\gamma}/A_1$, yet any disk of radius $\ell\in (\rho,1)$ intersects at most $A_2(\ell/\rho)^{\gamma}$ elements of $\mathcal{Q}$.

Then for any $\varepsilon>0$ there exists a set $J\subset [0,\pi)$ with the following properties:
\begin{enumerate}[(i)]
  \item $\mathcal{L}([0,\pi)\setminus J)\leq \varepsilon$.
  \item \label{separatedfamily} If $\theta\in J$, then for any family $\mathcal{Q}^{\prime}\subset\mathcal{Q}$ there exists a subcollection $\mathcal{Q}_1$ of $\mathcal{Q}^{\prime}$ of cardinality at least
  $$\delta\varepsilon\left(\frac{|\mathcal{Q}^{\prime}|}{|\mathcal{Q}|}\right)^{2}|\mathcal{Q}|,$$
  such that the elements of the family $\{\text{proj}_{\theta}(Q)\,:\,Q\in\mathcal{Q}_1\}$ are $\rho$-separated.
  \item $J$ is a finite union of open intervals.
\end{enumerate}\label{dpthm}
\end{thm}

\begin{proof}Throughout the course of the proof, $C_i$ denotes a constant which depends only on $A,A_1,A_2$ and $\gamma$.  Let $E$ be the union of elements in $\mathcal{Q}$, and let $\mu$ be normalized Lebesgue measure on $E$.  Consider the Riesz $1$-energy of $\mu$,
\begin{equation}\nonumber I_{1}(\mu)=\int\int |x-y|^{-1} d\mu(x) d\mu(y).\end{equation}  From  \cite[Proposition 7]{PeresShmerkin2009} we have that
\begin{equation}\nonumber I_{1}(\mu)\leq C_1\rho^{\gamma-1}.\end{equation}
Combining this with \cite[Theorem 9.7]{Mattila1995} yields
\begin{equation}\int_0^\pi \|d\mu_{\theta}/dx\|_2^2 d\theta \leq C_2 I_1(\mu)\leq C_3\rho^{\gamma-1},\label{density}\end{equation}
where $\mu_\theta$ is the push forward of $\mu$ under the map $P_{\theta}=R_{-\theta}\circ\text{proj}_\theta$, where $R_{-\theta}$ denotes a rotation of angle $-\theta$ about the origin.  Here $\| \cdot \|_2$ denotes the usual $L^2$ norm.

Let $f_{\theta}=d\mu_{\theta}/dx$ denote the density of $\mu_{\theta}$.  We claim that $\theta\mapsto \|f_\theta\|_2$ is continuous. By the dominated convergence theorem, it suffices to show that the density $f_{\theta}(x)$ is continuous for all $x$. Since
\begin{equation}
\nonumber f_\theta(x)=\sum_{Q\in\mathcal{Q}} f_{\theta,Q}(x),\end{equation}
where $f_{\theta,Q}$ is the density of the push forward of the measure $\mu|_{Q}$ under $P_{\theta}$, it is enough to verify that each $f_{\theta,Q}(x)$ is continuous, which follows from the convexity of $Q$.  Let

\begin{equation}
\nonumber J=\{\theta\,:\,\|f_\theta\|_2^2<\varepsilon^{-1} C_3 \rho^{\gamma-1}\}.
\end{equation}
From (\ref{density}) we deduce that $\mathcal{L}([0,\pi)\setminus J)\leq \varepsilon.$   Since $\|f_\theta\|_2$ depends continuously on $\theta$, $J$ is open and therefore a finite union of open intervals. Hence it remains to prove \eqref{separatedfamily}.

Fix a non-empty subfamily $\mathcal{Q}^{\prime}$ of $\mathcal{Q}$ and write $\eta=|\mathcal{Q}^\prime|/|\mathcal{Q}|.$  Let $\nu$ be the restriction of $\mu$ to the union of the elements of $\mathcal{Q}^{\prime}$, and let $\{\nu_\theta\}$ be the corresponding projections of $\nu$, and $\{f^{\prime}_\theta\}$ their densities.  As each element $Q\in\mathcal{Q}$ contains a disk of radius $A^{-1}\rho$, and is contained in a disk of radius $A\rho$, we have that $\mu(Q)/\mu(Q^{\prime}) \geq A^{-4}$ for all $Q,Q^{\prime}\in \mathcal{Q}$.  Hence, we deduce that $\nu$ has total mass at least $A^{-4}\eta$.  It is also clear that
\begin{equation}
\label{inequality1}
\|f^{\prime}_\theta\|_2^{2}\leq \|f_\theta \|_2^{2} <\varepsilon^{-1} C_3 \rho^{\gamma-1},\end{equation}
for all $\theta\in J$.
On the other hand, using the Cauchy-Schwarz inequality,
\begin{equation}
\label{inequality2}
A^{-4}\eta \leq \nu(\mathbb{R}^{2} )=\nu_\theta (\text{Supp}(\nu_\theta)) \leq \mathcal{L}(\text{Supp}(\nu_\theta))^{1/2}\| f^{\prime}_\theta \|_2.
\end{equation}
From (\ref{inequality1}) and (\ref{inequality2}) we deduce that if $\theta\in J$ then
\begin{equation}
\nonumber
\mathcal{L}(\text{Supp}(\nu_\theta)) \geq C_4 \varepsilon \eta^{2} \rho^{1-\gamma}.
\end{equation}
For such $\theta\in J$ we may select a set of points $E_{\theta}\subset \text{Supp}(\nu_\theta)$ of cardinality at least $\delta\varepsilon\eta^{2}|\mathcal{Q}|$, for which the elements are $(2A+1)\rho$-separated, where $\delta=\frac{C_4}{2 (2A+1)}$.  For each $x\in E_{\theta}$ choose an element $Q_x \in\mathcal{Q}^{\prime}$ for which $x\in P_{\theta}(Q_x)$. Since $P_{\theta}(Q_x)$ is contained in the interval $[x- A\rho, x+A\rho]$, we see that that the family $\{ P_\theta(Q_x): x\in E_\theta\}$ is $\rho$-separated. This concludes the proof, as $\{Q_x\}_{x\in E_{\theta}}$ has the desired properties. \end{proof}

\section{Proof of Proposition \ref{keyprop}}
\label{sec:keyprop}

Throughout this section, $\Lambda$ is the attractor of a homogeneous uniform-fibre Gatzouras-Lalley construction:
\[
\Lambda = \bigcup_{(i,j)\in D} S_{ij}(\Lambda),
\]
where $D=\{1,\ldots,m\}\times\{1,\ldots, n\}$, and $S_{ij}(x,y) = (f_{ij}(x), g_i(y))$, with $f_{i j}(x)=a x+ c_{i j}$ and $g_{i}(x)=b x+d_i$ for some $0<a<b<1$, $c_{ij}, d_i\in\R$.

We fix some notation. For a positive integer $N$, let  $$\Sigma_{N}=\{\sigma=\sigma_1\sigma_2 \sigma_3\cdots \,|\, 0\leq \sigma_k < N\,\text{for all k}\in\mathbb{N}\},$$ i.e. the set of one sided infinite strings over the alphabet $\{0,1,\ldots, N-1\}$, and let $\Sigma_{N}^{*}$ be the collection of all finite strings. We will also let
$$\Sigma_{N}^k=\{\sigma=\sigma_1\sigma_2 \sigma_3\cdots\sigma_k \,|\, 0\leq \sigma_k < N\,\text{for all }1\leq j\leq k\}.$$
  For a string $\sigma=\sigma_1 \sigma_2\ldots\in\Sigma_{N}$ and a natural number $k\in\mathbb{N}$, let $\sigma|_{k}=\sigma_1 \sigma_2 \cdots \sigma_k$ denote the truncation of $\sigma$ to its first $k$ symbols.  For a string $\sigma=\sigma_1 \sigma_2 \cdots \sigma_k\in\Sigma_N^{*}$ of finite length, we let $|\sigma|=k$ denote the length of the string.  For $\sigma,\eta\in\Sigma_{N}^{*}$ we denote their concatenation by $\sigma\eta=\sigma_1 \sigma_2\cdots \sigma_{|\sigma|} \eta_1 \eta_2\cdots \eta_{|\eta|}$.

This notation gives rise to a natural coding of the set $\Lambda$: let
\[
X_{k}=\{(\eta,\eta^\prime)\in\Sigma_{m}^{k}\times\Sigma_{n}^{k}\,\,:\,\, (\eta_i,\eta_i^\prime)\in D\,\,\text{ for }i=1,2,\ldots,k\},
\]
and write
$$f_{\sigma,\sigma^\prime}=f_{\sigma_1,\sigma^\prime_1}f_{\sigma_2,\sigma^\prime_2}\cdots f_{\sigma_k,\sigma^\prime_k},\,\,\,\,\, g_{\xi}=g_{\xi_1}g_{\xi_2}\cdots g_{\xi_{|\xi| }},$$
for $(\sigma,\sigma^{\prime})\in X_{k}$, and $\xi\in\Sigma_{m}^{*}$. For $\sigma\in\Sigma_{m}^{*}$, $(\xi,\sigma^\prime)\in X_k$ we set
\begin{equation}\nonumber Q(\sigma,\xi,\sigma^\prime)=f_{\xi,\sigma^\prime} [0,1]\times g_{\sigma}[0,1].\end{equation}
Then we may write
\begin{equation}\nonumber \Lambda=\bigcap_{k\geq 1} \bigcup_{(\sigma,\sigma^\prime)\in X_k} Q(\sigma,\sigma,\sigma^\prime).\end{equation}

It is easy to see that, in the homogeneous uniform fibre case,
\begin{equation} \label{dimhufc}
\gamma:=\dim_H \Lambda = \frac{\log m}{-\log b}+\frac{\log n}{-\log a}.
\end{equation}
If $\gamma\geq 1$, then for any $\varepsilon>0$ we can find a subset $\Lambda_{\varepsilon}\subset\Lambda$, which is also the invariant set of an irrational, homogeneous uniform fibre construction, and such that $1-\varepsilon<\dim_H(\Lambda_\varepsilon)<1$. Indeed, all we need to do is iterate the original iterated function system a large number of times, remove an appropriate subset of maps from the resulting system, and let $\Lambda_\varepsilon$ be the attractor of the reduced system. (It follows from \eqref{dimhufc} that by this procedure we can obtain sets of dimension arbitrarily close to any prefixed value not exceeding $\gamma$). Thus, by replacing $\Lambda$ with $\Lambda_\varepsilon $ and then letting $\varepsilon\rightarrow 0$, we can assume without loss of generality that $\gamma< 1$, and we will do so for the rest of the proof.

For a positive integer $k$, set $\ell(k)=\max\{k^\prime\,\,:\,\, b^k\leq a^{k^\prime}\}$, and $Z_{k}=a^{\ell(k)}b^{-k}\in [1,a^{-1})$.   Given $\xi\in\Sigma_{m}^{\ell(k)}$, $\xi^\prime\in\Sigma_m^{\ell(k)+1}$ we set
\begin{eqnarray}
\nonumber \mathcal{Q}_k(\xi)&=&\{Q(\sigma,\xi,\sigma^{\prime})\,\,:\,\,(\sigma,\sigma^\prime)\in\Sigma_m^{k}\times\Sigma_n^{\ell(k)}\},\\
\nonumber \tilde{\mathcal{Q}}_k(\xi^\prime)&=&\{Q(\sigma,\xi^\prime,\sigma^{\prime})\,\,:\,\,(\sigma,\sigma^\prime)\in\Sigma_m^{k}\times\Sigma_n^{\ell(k)+1}\}.\end{eqnarray}

We note that $(\xi,\sigma')\in X_{\ell(k)}$ and therefore $\mathcal{Q}_k(\xi)$ is well defined, and likewise is $\tilde{\mathcal{Q}}_k(\xi^\prime)$. Due to our choice of $\ell(k)$, and in particular the fact that $Z_k$ is bounded, the elements of $\mathcal{Q}_k(\xi)$ and $\tilde{\mathcal{Q}}_k(\xi^{\prime})$ are approximately squares, which project onto intervals of size $\approx b^{k}$, i.e. there exists a constant $c>1$ such that $c^{-1}\leq \text{diam}(\text{proj}_{\theta}(Q))/b^{k}\leq c$ for all $Q\in\mathcal{Q}_k(\xi).$

For each $\xi\in \Sigma_m^{\ell(k)}$ we may apply Theorem \ref{dpthm} to the family $\mathcal{Q}_k(\xi)$ to obtain a ``large'' set of ``good'' angles $J_\xi$. We remark that in the context of \cite[Theorem 1]{PeresShmerkin2009}, the sets $\mathcal{Q}_k(\xi)$ are all translates of each other thanks to the product structure. In particular, all $J_\xi$ are equal and hence there is a ``large'' set of ``good'' angles which is uniform in $\xi$. However, in our situation we have little information about the intersection of the sets $J_\xi$, or even the cardinality of the set $\{\xi\in\Sigma_m^{\ell(k)} \,\,:\,\,\theta\in J_\xi\}$ for a fixed $\theta\in [0,\pi)$. The reduction to the uniform horizontal fibre case and the modification of the discrete Marstrand Theorem given in Theorem \ref{dpthm} are needed in order to tackle this additional difficulty. Another step in this direction is the following lemma, which shows that there exists a ``large'' set of angles $J\subset [0,\pi)$, such that whenever $\theta\in J$, we have that $\theta\in J_\xi$ for a ``large'' number of $\xi\in\Sigma_m^{\ell(k)}$.

\begin{lemma}There exists a constant $\delta>0$ such that for any $k\in\mathbb{N}$, and $\varepsilon>0$, there exists a set $J\subset [0,\pi)$ with the following properties:
\begin{enumerate}[(i)]
  \item $\mathcal{L}([0,\pi)\setminus J)\leq \varepsilon^{1/2}.$
  \item If $\theta\in J$ then there exists a subcollection $\mathcal{B}\subset\Sigma_m^{\ell(k)}$ with the following properties:
   \begin{enumerate}
   \item The cardinality of  $\mathcal{B}$ is at least $(1-\varepsilon^{1/2})m^{\ell(k)}$.
   \item For any $\xi\in \mathcal{B}$, and any family $\mathcal{Q}^{\prime}\subset\mathcal{Q}_{k}(\xi)$, there exists a subcollection $\mathcal{Q}_1$ of $\mathcal{Q}^{\prime}$ of cardinality at least $$\delta\varepsilon \left(\frac{|\mathcal{Q}^{\prime}|}{|\mathcal{Q}_k(\xi)|} \right)^{2}m^{k\gamma},$$ for which the elements of the family $\{\text{proj}_\theta (Q)\,:\,Q\in\mathcal{Q}_1\}$ are $b^k$-separated.
   \end{enumerate}
  \item $J$ is a finite union of open intervals.
\end{enumerate}\label{dpthm2}
\end{lemma}

\begin{proof}
We first show that for a fixed $\xi\in\Sigma_m^{\ell(k)}$ the family $\mathcal{Q}_k(\xi)$ satisfies the hypothesis of Theorem \ref{dpthm}.   Each $Q\in\mathcal{Q}_{k}(\xi)$ has size $Z_{k}b^k\times b^k$, thus taking $A=\sqrt{1+a^{-2}}$ implies that $Q$ is contained inside some ball of radius $Ab^k$ with centre in $Q$.  Similarly, $Q$ contains a ball of radius $A^{-1}b^k$.  In addition,
\begin{eqnarray}
\nonumber |\mathcal{Q}_{k}(\xi)| & = & m^k n^{\ell(k)} \\
\nonumber  & > & n^{-1} \left( n^{\log_a b} m \right)^{k} \\
\label{lowercard}  & = &  n^{-1} b^{-k\gamma}=A_1^{-1}b^{-k\gamma}.
\end{eqnarray}
Finally fix $j<k$, and let $Q=Q(\sigma,\xi|_{\ell(j)},\sigma^{\prime})\in\mathcal{Q}_{j}(\xi|_{\ell(j)})$ then
\begin{eqnarray}
\nonumber |\{Q(\eta,\xi,\eta^{\prime})\in \mathcal{Q}_{k}(\xi)\,:\,Q(\eta,\xi,\eta^{\prime})\subset Q\}| & = & |\{ (\eta,\eta^{\prime})\in \Sigma_m^k\times \Sigma_n^{\ell(k)} \,:\, \eta|_{\ell(j)}=\sigma,\,\,\eta^{\prime}|_{j}=\sigma^{\prime} \}| \\
\nonumber  & = & m^{k-j} n^{\ell(k)-\ell(j)} \\
 & < &  n b^{-(k-j)\gamma}. \label{balls}
\end{eqnarray}  Any ball $B$ of radius $b^{j}$ may intersect at most 9 elements of $\mathcal{Q}_{j}(\xi|_{\ell(j)})$. Combining this with (\ref{balls}) and setting $A_2=9n$, we see that $B$ may intersect at most $A_2 b^{-(k-j)\gamma}$ elements of $\mathcal{Q}_k(\xi)$.  It is worth noting that the constants $A,A_1,A_2,\gamma$ have no dependence on the choice of $\xi$.

The above argument combined with Theorem \ref{dpthm} implies the following: there is $\delta>0$ such that for any $\varepsilon>0$, $k\in\mathbb{N}$ and $\xi\in\Sigma_m^{\ell(k)}$, there exists a set of angles $J_{\xi}\subset [0,\pi)$ satisfying:
\begin{enumerate}[(i)]
  \item $\mathcal{L}([0,\pi)\setminus J_{\xi})\leq \varepsilon.$
  \item If $\theta\in J_\xi$, then for any family $\mathcal{Q}^{\prime}\subset\mathcal{Q}_k(\xi)$, there exists a subcollection $\mathcal{Q}_1$ of $\mathcal{Q}^{\prime}$ of cardinality at least
  $$\delta\varepsilon\left(\frac{|\mathcal{Q}^{\prime}|}{|\mathcal{Q}_k(\xi)|}\right)^{2}|\mathcal{Q}_k(\xi)|,$$
  such that the elements of the family $\{\text{proj}_{\theta}(Q)\,:\,Q\in\mathcal{Q}_1\}$ are $b^k$-separated.
  \item $J_\xi$ is a finite union of open intervals.
\end{enumerate}

Furthermore combining (\ref{lowercard}) with property (ii) above, we may take the cardinality of $\mathcal{Q}_1$ to be at least $A_1^{-1}\delta\varepsilon \left(|\mathcal{Q}^{\prime}|/|\mathcal{Q}_k(\xi)|\right)^{2}b^{-k\gamma}.$

For $\theta\in [0,\pi)$, let  $\Xi_{\theta}=\{\xi\in\Sigma_m^{\ell(k)}\,:\, \theta\in J_{\xi}\}$, then
\begin{equation}\label{reverse} \int_{0}^{\pi} |\Sigma_m^{\ell(k)} \setminus \Xi_{\theta}| d\theta =\sum_{\xi\in \Sigma_m^{\ell(k)}} \mathcal{L}([0,\pi)\setminus J_{\xi}) < m^{\ell(k)}  \varepsilon.\end{equation}
We set
\begin{equation}
\nonumber J=\{\theta\in [0,\pi) \,:\, |\Xi_\theta|>(1-\varepsilon^{1/2}) m^{\ell(k)} \}.
\end{equation}
Then from (\ref{reverse}) we deduce $\mathcal{L}([0,\pi)\setminus J)\leq \varepsilon^{1/2}$.  Finally, we may express $J$ in the form
\begin{equation}\nonumber
J= \bigcup_{|B|>(1-\varepsilon^{1/2})m^{\ell(k)}} \bigcap_{\xi\in B} J_{\xi}\end{equation}
which is easily seen to be a finite union of open intervals.  This completes the proof.
\end{proof}

Given $\tau\in\R$, let $\Pi_\tau(x,y)= a^{-\tau}x+y$ and $\widetilde{\Pi}_\tau(x,y) = -a^{-\tau}x+y$. Note that for a fixed $\theta\in (0,\pi/2)$, the set $\text{proj}_{\theta}(\Lambda)$ is affinely equivalent to the set $\Pi_{\tau}(\Lambda)$, where $\tau=\log(\tan (\theta))/\log a$.  Similarly, for $\theta\in (\pi/2,\pi)$ the set $\text{proj}_{\theta}(\Lambda)$ is affinely equivalent to $\widetilde{\Pi}_\tau(\Lambda)$ for $\tau=\log(-\tan(\theta))/\log(a).$  For technical reasons, it is convenient to work with the maps $\Pi_\tau$ in the place of the orthogonal projections $\text{proj}_\theta$.  Thus, as affine bijections preserve Hausdorff dimension, it suffices to show that
\[
\dim_H (\Pi_\tau \Lambda),\,\dim_H(\tilde{\Pi}_\tau \Lambda )\geq \gamma\quad\textrm{for any }\tau\in\mathbb{R}.
\]
The proofs for $\Pi_\tau$ and $\widetilde{\Pi}_\tau$ are analogous, and accordingly we will present the proof only for $\Pi_\tau$, corresponding to $\theta\in (0,\pi/2)$.

Since the map $\theta\mapsto\log(\tan (\theta))/\log a$ is a $C^1$ injection on any compact subset of $(0,\pi/2)$, a reparametrization of Lemma \ref{dpthm2} yields the following:
\begin{cor}Given $\tau\in\mathbb{R}$, there exist constants $\delta^{\prime},L>0$ such that for any $k\in\mathbb{N}$ and $\varepsilon>0$ there is a set $\tilde{F}\subset [\tau,\tau+1)$ with the following properties:
\begin{enumerate}[(i)]
  \item $\mathcal{L}([\tau,\tau+1)\setminus \tilde{F})\leq L\varepsilon^{1/2}.$
  \item If $t\in \tilde{F}$, then there exists a subcollection $\mathcal{B}_t\subset\Sigma_m^{\ell(k)}$ of cardinality at least $(1-\varepsilon^{1/2}) m^{\ell(k)}$, such that for any $\xi\in \mathcal{B}_t$ and family $\mathcal{Q}^{\prime}\subset\mathcal{Q}_{k}(\xi)$, there exists a subcollection $\mathcal{Q}_1$ of $\mathcal{Q}^{\prime}$ of cardinality at least $$\delta^{\prime}\varepsilon \left(\frac{|\mathcal{Q}^{\prime}|}{|\mathcal{Q}_k(\xi)|} \right)^{2}b^{-k\gamma},$$ for which the elements of the family $\{\Pi_t (Q)\,:\,Q\in\mathcal{Q}_1\}$ are $b^k$-separated.
  \item $\tilde{F}$ is a finite union of open intervals.
\end{enumerate}\label{dpthm3}
\end{cor}

Property (ii) of Corollary \ref{dpthm3} holds for the families $\tilde{\mathcal{Q}}_k(\xi^{\prime})$ as each element is contained in an element of $\mathcal{Q}_k(\xi)$ with $\xi^{\prime}|_{\ell(k)}=\xi$.  For $t\in [\tau,\tau+1)$, let
\[
\tilde{\mathcal{B}}_t=\{\xi\xi^{\prime}\,:\,\xi\in \mathcal{B}_t,\,\xi^{\prime}\in\Sigma_m\}\subset \Sigma_m^{\ell(k)+1}
\]
be the collection of strings obtained by adjoining an element of $\Sigma_m$ to those of $\mathcal{B}_{t}$.

We now construct a branching process driven by an irrational rotation, which we will then use to construct a subset $\Lambda_\tau$ of $\Pi_\tau \Lambda$, satisfying $\dim_H \left(\Lambda_\tau \right)\geq \gamma.$  Define a map $T:[0,1)\rightarrow [0,1)$ by $T(x)=x+\alpha \mod 1$, where $\alpha=\log (Z_k)/-\log a=k \log_a b - \ell(k)\in [0,1)$.   We may choose a sequence $\{e_j\}_j$ to satisfy
\begin{equation}\label{transformation}T^j(0)=\frac{\log (a^{e_j} b^{-jk})}{-\log a},\end{equation}
i.e. set $e_1=\ell(k)$, and for $j>1$ let
\begin{equation}
\nonumber\label{sequence}
e_{j} =\begin{cases}
      & e_{j-1}+\ell(k) \qquad \text{ if }\,\,a^{e_{j-1}+\ell(k)+1}<b^{jk}  \\
      & e_{j-1}+\ell(k)+1\,\,\text{ otherwise}.
\end{cases}
\end{equation}
For $j\in\mathbb{N}$ let
\begin{equation}\nonumber
\Gamma_k(j)  =  \begin{cases}
      & \mathcal{B}_{t}\,\, \text{ if }\,\,t=\tau+T^{j}(0)\in \tilde{F}\,\,\text{ and }\,\, e_{j+1}-e_j=\ell(k), \\
      & \tilde{\mathcal{B}}_{t}\,\, \text{ if }\,\,t=\tau+T^{j}(0)\in \tilde{F}\,\,\text{ and }\,\, e_{j+1}-e_j=\ell(k)+1, \\
      &  \Sigma_m^{e_{j+1}-e_j}\,\,\text{ otherwise}
\end{cases}\end{equation}
where $t=\tau+T^j(0)$.
In which case, by Corollary \ref{dpthm3},
\begin{equation}
\label{countingstuff}
|\Gamma_k(j)|>(1-\varepsilon^{1/2})m^{e_{j+1}-e_j}.\end{equation}

We now introduce a function $s:\mathbb{N}\rightarrow\mathbb{N}$ which relates the two quantities $e_j$ and $jk$.  For a positive integer $j$, let $s(j)$ be the unique integer such that $e_{s(j)-1}<jk \leq e_{s(j)}$, then from (\ref{transformation}) we see
\begin{equation}\label{sj} j\frac{\log a}{\log b} \leq s(j) < j\frac{\log a}{\log b}+\frac{\log a}{k\log b}+1.  \end{equation}
For $j\in\mathbb{N}$ we let
\begin{eqnarray}
\nonumber \Delta_k(j) & = & \{ \eta\in\Sigma_m^{jk-e_{s(j)-1}}\,\,:\,\, \eta\eta^\prime\in\Gamma_k(s(j)-1)\,\,\text{ for some }\,\eta^\prime\in\Sigma_m^{e_{s(j)}-jk}\}, \\
\nonumber \Delta_k^\prime (j) & = &  \{\eta^\prime\in\Sigma_m^{e_{s(j)}-jk}\,\,:\,\, \eta\eta^\prime\in\Gamma_k(s(j)-1)\,\,\text{ for some }\,\eta\in\Sigma_m^{jk-e_{s(j)-1}}\}.\end{eqnarray}
In the event that $jk=e_{s(j)}$, we take $\Delta_{k}^\prime(j)$ to contain the empty word, otherwise, both $\Delta_k(j)$ and $\Delta_k^\prime(j)$ are non-empty and, moreover, from (\ref{countingstuff}) we deduce that
\begin{eqnarray}
\nonumber |\Delta_k(j)| &>& (1-\varepsilon^{1/2})m^{jk-e_{s(j)-1}},\\
\nonumber |\Delta^{\prime}_k(j)| &>& (1-\varepsilon^{1/2})m^{e_{s(j)}-jk}.\end{eqnarray}
For $\eta\in\Sigma_m^{jk-e_{s(j)-1}}$, set
\begin{equation}\nonumber \phi_j(\eta)=|\{\eta^\prime\in\Delta^\prime_k(j)\,\,:\,\,\eta\eta^\prime\in\Gamma_k(s(j)-1)\}|.\end{equation}
Then by letting
\begin{equation}\nonumber \Theta(j):=\left\{\eta\in\Delta_k(j)\,\,:\,\,\phi_j(\eta)> \frac{1}{2}m^{e_s(j)-jk}\right\},\end{equation}
we deduce that
\begin{eqnarray}
\nonumber |\Sigma_m^{e_{s(j)}-e_{s(j)-1} }\setminus \Gamma_k(s(j)-1)| & = & \sum_{\eta\in \Sigma_m^{jk-e_{s(j)-1}} }\left( m^{e_{s(j)}-jk}-\phi_j(\eta)\right)\\
\nonumber &>&  \sum_{\eta\in \Sigma_m^{jk-e_{s(j)-1}} \setminus\Theta(j) }\left(  m^{e_{s(j)}-jk}-\phi_j(\eta)\right) \\
\nonumber &>& \frac{1}{2} m^{e_{s(j)}-jk} |\Sigma_m^{jk-e_{s(j)-1}} \setminus\Theta(j)|.
\end{eqnarray}
Combining this with (\ref{countingstuff}) we see that $|\Theta(j)|>(1-2\varepsilon^{1/2})m^{jk-e_{s(j)-1}}.$  For $\eta\in\Theta(j)$, let
\begin{equation}\nonumber \Theta^\prime(j,\eta)=\{\eta^\prime\,\,:\,\, \eta\eta^\prime\in\Gamma_k(s(j)-1)\},\end{equation}
and
\begin{equation}\nonumber
G_{k}(j,\eta)= \{\eta^{\prime}\xi_1\xi_2\cdots\xi_{s(j+1)-s(j)-1}\eta^{\prime\prime}\,\,:\,\, \eta^{\prime}\in\Theta^\prime(j,\eta),\,\,\eta^{\prime\prime}\in\Theta(j+1),\,\,\xi_i\in\Gamma_k(i+s(j)-1)\},\end{equation}
Then from (\ref{countingstuff}), (\ref{sj}) and the definitions of $\Delta_k(j)$ and $\Delta_k^{\prime}(j)$, we see that if $\varepsilon<1/9$, then
\begin{eqnarray}
\nonumber |G_k(j,\eta)| &=&  |\Theta(j+1)||\Theta^{\prime}(j,\eta)|\prod_{i=s(j)}^{s(j+1)-2} |\Gamma_k(i)| \\
\nonumber &>& \frac{1}{2}(1-2\varepsilon^{1/2})(1-\varepsilon^{1/2})^{s(j+1)-s(j)-1}m^k \\
\nonumber &>& \frac{1}{6}(1-\varepsilon^{1/2})^{s(j+1)-s(j)-1}m^k \\
 &>& \frac{1}{6}\left(\frac{2}{3}\right)^{2\frac{\log a}{\log b}} m^k. \label{G}  \end{eqnarray}
For $\xi\in\Sigma_m^{e_{j+1}-e_j}$ and $\eta\in\Theta(j)$ let
\begin{eqnarray}\nonumber
M_k(j,\eta)&=& G_k(j,\eta)\times\Sigma_n^{e_{j+1}-{e_j}},\\
\nonumber\mathcal{U}_k(j,\xi,\eta)&=&\{Q(\sigma,\xi,\sigma^{\prime})\,\,:\,\, (\sigma,\sigma^{\prime})\in M_k(j,\eta) \}.\end{eqnarray}
Then for any $\xi\in\Sigma_m^{e_{j+1}-e_j}$ and $\eta\in\Theta(j)$ such that $e_{j+1}-e_j=\ell(k)$, (\ref{G}) yields
\begin{eqnarray}
\nonumber |\mathcal{U}_k(j,\xi,\eta)| & = &n^{\ell(k)} |G_k(j,\eta)|  \\
\nonumber  & > &    \frac{1}{6}\left(\frac{2}{3}\right)^{2\frac{\log a}{\log b}} n^{\ell(k)} m^k \\
\nonumber & > &  c |\mathcal{Q}_k(\xi)|,\end{eqnarray}
for some constant $c$, which is independent of $k$ and $\varepsilon$.  Similarly, if $e_{j+1}-e_j=\ell(k)+1$ we have
\begin{equation}
\nonumber  |\mathcal{U}_k(j,\xi,\eta)|>c  |\tilde{\mathcal{Q}}_k(\xi)|.\end{equation}
Thus if $t=T^{j}(0)+\tau\in\tilde{F}$, $\xi\in \Gamma_k(j)$ and $\eta\in\Theta(j)$, by Corollary \ref{dpthm3} we have that there exists a subcollection $\mathcal{U}^{\prime}_k(j,\xi,\eta)$ of $\mathcal{U}_k(j,\xi,\eta)$, of cardinality at least $c^{2} \delta^{\prime}\varepsilon b^{-k\gamma}$, and for which the elements of the family $\{\Pi_t (Q)\,:\,Q\in\mathcal{U}^{\prime}_k(j,\xi,\eta)\}$ are $b^{k}$-separated.  We let $M^{\prime}_k(j,\xi,\eta)$ denote the subset of $M_k(j,\eta)$ for which
\begin{equation}\nonumber
\mathcal{U}_k^{\prime}(j,\xi,\eta)=\{Q(\sigma,\xi,\sigma^{\prime})\,\,:\,\,(\sigma,\sigma^{\prime})\in M^{\prime}_k(j,\xi,\eta)\}.\end{equation}

\begin{lemma}There exists a rooted tree, $\mathcal{R}$, with vertices labeled by elements of $\Sigma_m^{jk}\times\Sigma_n^{e_j}$ for some $j\in\mathbb{N}$, with the following properties: Let us denote the elements of $j$'th level by $\mathcal{R}_j$. Then:
\begin{description}
  \item[A] If $(\sigma,\sigma^{\prime})$ is the parent of $(\eta,\eta^{\prime})$ then $(\eta|_{|\sigma|},\eta^{\prime}|_{|\sigma^{\prime}|})=(\sigma,\sigma^{\prime}).$
  \item[B] If $(\sigma,\sigma^{\prime})\in\mathcal{R}_{j}$ then $Q(\sigma,\sigma|_{|\sigma'| },\sigma^{\prime})$ has size $a^{-T^j(0)}b^{jk}\times b^{jk}.$
  \item[C] If $(\sigma,\sigma^{\prime})\in\mathcal{R}_j$ then
  \begin{eqnarray} \nonumber \xi_i(\sigma)&:=&\sigma_{e_i+1}\sigma_{e_i+2}\cdots\sigma_{e_{i+1}}\in\Gamma_k(i)\,\,\,\text{ for }\,\, j\leq i \leq s(j)-2,\\
  \nonumber \eta(\sigma^{\prime})&:=&\sigma_{e_{s(j)-1}+1}\sigma_{e_{s(j)-1}+2}\cdots\sigma_{jk}\in\Theta(j).\end{eqnarray}
  \item[D] The elements of the set $\{\Pi_{\tau}Q(\sigma,\sigma|_{|\sigma^\prime|},\sigma^\prime)\,:\,(\sigma,\sigma^{\prime})\in\mathcal{R}_{j}\}$ are disjoint and $b^{jk}$-separated.
  \item[E] Each element of $\mathcal{R}_{j}$ has the same number of offspring $C_j$, moreover
  \begin{eqnarray}
  \nonumber  T^{j}(0)+\tau\in \tilde{F} \,\,\,&\implies& C_j > c^2 \delta^{\prime}\varepsilon b^{-k\gamma},\\
  \nonumber  T^{j}(0)+\tau\not\in \tilde{F} \,\,\,&\implies& C_j=1.\end{eqnarray}
 \end{description}\label{tree}
\end{lemma}

\begin{proof}
At each stage of the construction we define the offspring $C(\sigma,\sigma^{\prime})$ of the element $(\sigma,\sigma^{\prime})\in\mathcal{R}_{j}$ to be the concatenation of $(\sigma,\sigma^{\prime})$ with elements of $\mathcal{U}_{k}(\xi,j)$, where $\xi=\xi_j(\sigma)$.   Accordingly, we require $(\sigma,\sigma^{\prime})$ of $\mathcal{R}_{j}$ to be such that $|\sigma|-|\sigma^{\prime}|>\ell(k)$. Therefore we will start the induction at $j_0\in\mathbb{N}$, where $\ell(k)<j_0 k-e_{j_0}.$

To construct $\mathcal{R}_{j_0}$, let $(\sigma,\sigma^{\prime})$ be any element of $ \Sigma_m^{e_{j_0}}\times\Sigma_n^{e_{j_0}}$, and for $i=j_0,j_0+1,\ldots,s(j_0)-2$ let $\xi_i\in\Gamma_k(i)$. Then set

$$\mathcal{R}_{j_0}=\{(\sigma\xi_{j_0}\xi_{j_0+1}\cdots\xi_{s(j_0)-2}\eta,\sigma^{\prime})\},$$ where $\eta$ is any element of $\Theta(j_0)$.  Properties \textbf{A}, \textbf{D} and \textbf{E} follow trivially.  To see \textbf{B}, the rectangle $Q(\sigma\xi_{j_0}\xi_{j_0+1}\cdots\xi_{s(j_0)-2}\eta,\sigma,\sigma^{\prime})$ has size $a^{e_{j_0}}\times b^{j_0k}=a^{-T^{j_0}(0)} b^{j_0 k}\times b^{j_0k}$.  By construction $\xi_i(\sigma\xi_{j_0}\xi_{j_0+1}\cdots\xi_{s(j_0)-2}\eta)=\xi_{i}\in\Gamma_k(i)$,  for $i=j_0,j_0+1,\ldots,s(j_0)-2$, and $\eta(\sigma\xi_{j_0}\xi_{j_0+1}\cdots\xi_{s(j_0)-2}\eta)=\eta\in\Theta(j)$, showing \textbf{C}.

Now suppose that for $j>j_0$, $\mathcal{R}_j$ has been defined and properties \textbf{A}-\textbf{E} hold.  We now construct $\mathcal{R}_{j+1}$.  Fix an element $(\sigma,\sigma^\prime)\in\mathcal{R}_j$, and let $\xi=\xi_j(\sigma)$, $\eta=\eta(\sigma)$.  In the case that $T^{j}(0)+\tau\in\tilde{F}$ let the offspring $C(\sigma,\sigma^{\prime})$ of $(\sigma,\sigma^{\prime})$ be
\begin{equation}\nonumber
C(\sigma,\sigma^{\prime})=\{ (\sigma\nu,\sigma^{\prime}\nu^{\prime})\,\,:\,\, (\nu,\nu^{\prime})\in M_k^{\prime}(j,\xi,\eta)\},\end{equation}
otherwise $T^j(0)+\tau\not\in\tilde{F}$, in which case we let $C(\sigma,\sigma)=\{ (\sigma\nu,\sigma^{\prime}\nu^{\prime})\}$ where $(\nu,\nu^{\prime})$ is any element of $M_{k}(j,\eta)$.  Property \textbf{A} is clear.  Any element of $\mathcal{R}_{j+1}$ has size
$$a^{e_{j+1}}\times b^{(j+1)k}=a^{-T^{j+1}(0)} b^{(j+1)k}\times b^{(j+1)k},$$ which shows \textbf{B}.  If $(\sigma\nu,\sigma^{\prime}\nu^{\prime})\in C(\sigma,\sigma^{\prime})$ then as $\mathcal{R}_j$ satisfies property \textbf{C}, we have $\xi_i(\sigma\nu)=\xi_i(\sigma)\in\Gamma_j(i)$ for $i=j+1,j+2,\ldots,s(j)-2$ and $\eta=\eta(\sigma)\in\Theta(j)$.  Further, $(\nu,\nu^{\prime})\in M^\prime_k(j,\xi,\eta)$ and so $\xi_i(\sigma\nu)\in\Gamma_k(i)$ for $i=s(j),s(j)+1,\ldots, s(j+1)-2$, and $\eta(\sigma\nu)\in\Theta(j+1)$.  Finally, $\xi_{s(j)-1}(\sigma\nu)=\eta\eta^\prime$, where $\eta\in\Theta(j)$ and $\eta^\prime=\nu_1\nu_2\cdots\nu_{s(j)-jk}\in\Theta^\prime(j+1,\eta)$, which from the definitions of $\Theta(j)$ and $\Theta^\prime(j+1,\eta)$ can be seen to be an element of $\Gamma_k(s(j)-1)$, and this shows \textbf{C}.

\begin{figure}
\begin{center}
\includegraphics[width=0.8\textwidth]{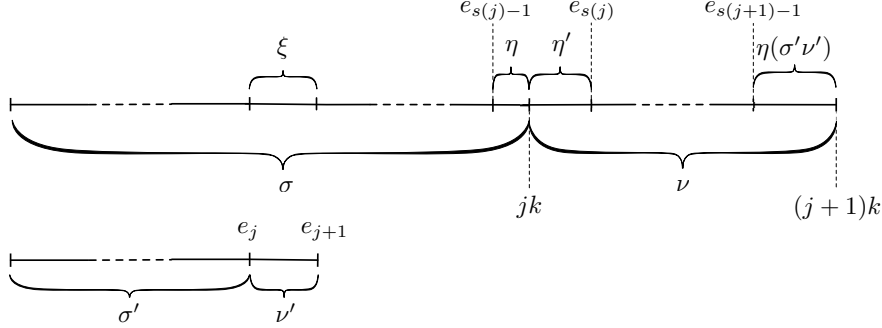}
\end{center}
\caption{A diagram depicting the construction of an element of
$\mathcal{R}_{j+1}$.}
\end{figure}

If $T^j(0)+\tau \not\in \tilde{F}$ then property \textbf{D} is trivial, so assume that $T^j(0)+\tau\in \tilde{F}$.   It is sufficient to show that if $(\sigma\nu_1,\sigma^{\prime}\nu_1^{\prime}), (\sigma\nu_2,\sigma^{\prime}\nu_2^{\prime})$ are offspring of $(\sigma,\sigma^{\prime})\in\mathcal{R}_{j}$ then the projections of $Q(\sigma\nu_1,\sigma\nu_1|_{e_{j+1}},\sigma^{\prime}\nu_1^{\prime})$ and $Q(\sigma\nu_2,\sigma\nu_2|_{e_{j+1}},\sigma^{\prime}\nu_2^{\prime})$ are $b^{(j+1)k}$-separated.   By translating $Q(\sigma,\sigma|_{e_j},\sigma^{\prime})$, if necessary, we may assume that $Q(\sigma,\sigma|_{e_j},\sigma^{\prime})=[0,a^{-T^{j}(0)}b^{j k}]\times [0,b^{j k}]$, and therefore
\begin{eqnarray}
\nonumber \Pi_{\tau}(Q(\sigma\nu_i,\sigma\nu_i|_{e_{j+1}},\sigma^{\prime}\nu_i^{\prime})) & = & \{a^{-\tau}x+y\,:\,(x,y)\in Q(\sigma\nu_i,\sigma\nu_i|_{e_{j+1}},\sigma^{\prime}\nu_i^{\prime}))\} \\
\nonumber  & = & b^{j k}\{a^{-\tau-T^{j}(0)}x+y\,:\,(x,y)\in Q(\nu_i,\xi,\nu_i^{\prime})\} \\
\nonumber  & = & b^{j k}\Pi_{\tau+T^{j}(0)}Q(\nu_i,\xi,\nu_{i}^{\prime}),
\end{eqnarray}
where the fact that $\xi=\xi_j(\sigma)=\xi_j(\sigma\nu_i)$ was used on the penultimate line.  By property (ii) of Corollary \ref{dpthm3}, the family $\{\Pi_{\tau+T^{j}(0)}(Q(\nu_i,\xi,\nu^\prime_i))\}$ is $b^k$-separated; it follows that the family
\[
\{\Pi_{\tau}(Q(\sigma\nu_i,\sigma\nu_i|_{e_{j+1}},\sigma^{\prime}\nu_i^{\prime}))\}
\]
is $b^{(j+1)k}$-separated, as desired.  Property \textbf{E} follows from property (ii) of Corollary \ref{dpthm3} and the fact that $\mathcal{R}_j$ satisfies property \textbf{C}. \end{proof}

We  can now complete the proof of Proposition \ref{keyprop}.
\subsection*{Proof of Proposition \ref{keyprop}}
By assumption $\log a/\log b$ is irrational, which guarantees the irrationality of $\alpha$. Thus, by Weyl's equidistribution theorem, the sequence $\{T^{j}(0)\}_j $ is equidistributed in $[0,1)$.  Applying this to the set $F=\{x\in [0,1)\,:\, \tau+x\in\tilde{F}\}$ yields

\begin{equation}
\label{weyl}
\lim_{j\rightarrow\infty} \frac{1}{j-j_0}|\{j_0< i\leq j\,:\,T^{i}(0)\in F\}|=\mathcal{L}^{1}(F)\geq 1- L\varepsilon^{1/2}.
\end{equation}
Let $$\Lambda_{\tau}=\bigcap_{j>j_0} \bigcup_{(\sigma,\sigma^{\prime})\in\mathcal{R}_j} \Pi_\tau (Q(\sigma,\sigma|_{e_j},\sigma^\prime)),$$
where $\mathcal{R}_j$ is the j'th level of the tree $\mathcal{R}$ guaranteed by Lemma \ref{tree}.  By property \textbf{A} this set is a countable intersection of nested compact sets, so $\Lambda_{\tau}$ is compact and nonempty.  From the construction of $\Lambda_\tau$ it is clear that $\Lambda_{\tau}\subset \Pi_{\tau}(\Lambda)$.

We will estimate the dimension of $\Lambda_{\tau}$, and will do so in a standard way by means of constructing a natural measure supported on $\Lambda_\tau$.  Let $\mu_{\tau}$ be the probability measure which assigns the same mass $|\mathcal{R}_{j}|^{-1}$ to each of the intervals $\Pi_{\tau}(Q(\sigma,\sigma|_{e_j},\sigma^\prime))$ for $(\sigma,\sigma^{\prime})\in\mathcal{R}_j$.  This measure is well defined thanks to properties \textbf{D} and \textbf{E}, and is supported on $\Lambda_{\tau}$.
Let $x\in \Lambda_{\tau}$.  Because of \textbf{D}, for $j>j_0$ the interval $(x-b^{jk}/2,x+b^{jk}/2)$ intersects exactly one interval $\Pi_{\tau}(Q(\sigma,\sigma|_{e_j},\sigma^\prime))$ with $(\sigma,\sigma^{\prime})\in\mathcal{R}_{j}$, and thus
\begin{equation}\label{mass}\mu_{\tau}(x-b^{jk}/2,x+b^{jk}/2)\leq \mu_{\tau}(\Pi_{\tau}(Q(\sigma,\sigma|_{e_j},\sigma^\prime)))=|\mathcal{R}_{j}|^{-1}.\end{equation}
On the other hand, using \textbf{E} once again we obtain
\begin{equation}\label{card}\log |\mathcal{R}_j|\geq |\{ j_0 \leq i < j \,:\, T^{i}(0)\in F\}| \log (c^2\delta^{\prime} \varepsilon b^{-k\gamma}).\end{equation}
Hence we learn from (\ref{weyl}), (\ref{mass}) and (\ref{card}), that if $j^{\prime}$ is large enough then for all $j>j^{\prime}$ we have
$$\log\mu_{\tau}(x-b^{jk}/2,x+b^{jk}/2)\leq -(j-j_0)(1-2L\varepsilon^{1/2})\log(c^2\delta^{\prime} \varepsilon b^{-k\gamma}).$$
We then deduce that
\begin{eqnarray}
\nonumber \underline{\dim}_{loc}(\mu_{\tau},x) & := & \liminf_{j\rightarrow\infty} \frac{\log(\mu_{\tau}(x-b^{jk}/2,x+b^{jk}/2))}{jk\log (b)-\log 2} \\
\nonumber & \geq & \liminf_{j\rightarrow\infty} \frac{(j-j_0)(1-2L\varepsilon^{1/2})\log(c^2\delta^{\prime} \varepsilon b^{-k\gamma})}{\log 2-jk\log b} \\
\nonumber & = & \frac{(1-2L\varepsilon^{1/2})\log(c^2\delta^{\prime} \varepsilon b^{-k\gamma})}{-k\log b},
\end{eqnarray}
for any $x\in \Lambda_\tau$.  From the Mass Distribution Principle (see e.g. \cite[Proposition 2.3]{Falconer1997}) we conclude that
$$\dim_H(\Pi_{\tau}(\Lambda))\geq \dim_H (\Lambda_{\tau}) \geq  \frac{(1-2L\varepsilon^{1/2})\log(c^2\delta^{\prime} \varepsilon b^{-k\gamma})}{-k\log b}.$$
The right hand side can be made arbitrarily close to $\gamma$ by letting $k\rightarrow\infty$ and then $\varepsilon\rightarrow 0$. \qed

\section{Proof of Theorem \ref{thm1}}
\label{sec:main}

To deduce Theorem \ref{thm1} from Proposition \ref{keyprop} we show that any system satisfying the hypotheses of Theorem 1 can be arbitrarily well approximated from inside, in terms of dimension, by a homogeneous uniform fibre system. Moreover the matrix $A=\begin{pmatrix}a&0\\0&b\end{pmatrix}$ in the homogeneous uniform fibre system needs to satisfy $\frac{\log a}{\log b}\notin\Q$. We start with two straightforward lemmas concerning when ratios of real numbers are rational.
\begin{lemma}\label{rat1}
If $x,y,r,t$ are nonzero real numbers and $\frac{x}{y},\frac{x+r}{y+t},\frac{x+2r}{y+2t}\in\Q$ then $\frac{r}{t}\in\Q$
\end{lemma}
\begin{proof}
This follows from straightforward calculations.
\end{proof}
\begin{lemma}\label{rat2}
Consider $\R$ as a vector space over $\Q$ and let $\{a_1,\ldots,a_n\}\subset\R$ be linearly independent. If there exist $p_1,\ldots,p_n,q_1,\ldots q_n\in \Q\backslash\{0\}$ such that
$$\frac{p_1a_1+\ldots+p_na_n}{q_1a_1+\ldots+q_na_n}=\frac{x}{y}$$
where $x,y\in\Z\backslash\{0\}$
then for $1\leq i\leq n$ we have that $\frac{p_i}{q_i}=\frac{x}{y}$.
\end{lemma}
\begin{proof}
This follows by routine linear algebra.
\end{proof}
We are now ready to prove the key lemma:
\begin{lemma}\label{approx}
Let
 $\{S_{i,j}\}$ be a Gatzouras-Lalley or Bara\'{n}ski iterated function system and $\Lambda$ be the attractor. Given $\varepsilon>0$, there exists an iterated function system $\{A\cdot x + y_i\}_{i=1}^N$, where $A=\begin{pmatrix}a&0\\0&b\end{pmatrix}$ is a diagonal matrix and the system has uniform horizontal fibres, such that the associated invariant set $\Lambda^\prime$ satisfies $\dim_H \Lambda^\prime > \dim_H \Lambda - \varepsilon$ and $\Lambda^\prime\subset \Lambda$. Moreover, if the iterated function system $\{ S_{i,j}\}$ is of irrational type, we can choose $A$ such that $\frac{\log a}{\log b}\notin\Q$.
\end{lemma}
\begin{proof} We just give the proof in the Gatzouras-Lalley case since the Bara\'{n}ski case is proved in a similar way. Let $\textbf{p}=(p_1,p_2,\ldots, p_m)$ be the unique probability vector for which the supremum in (\ref{dimH}) is attained, that is
\begin{equation}
\nonumber\dim_H \Lambda =\frac{\sum_{i=1}^m p_i \log p_i}{\sum_{i=1}^m p_i\log b_i}+t(\textbf{p})\end{equation}
where $t(\textbf{p})$ is the unique real number satisfying
\begin{equation}
\nonumber\sum_{i=1}^m p_i \log \left(\sum_{j=1}^{n_i} a_{i j}^{t(\textbf{p})}\right)=0.
\end{equation}
For $(i,j)\in D=\{ (i,j) \,\,:\,\, 1\leq i \leq m,\, 1\leq j \leq n_i\}$ we let
\begin{equation}
\nonumber q_{i j}= \frac{ p_i a_{i j}^{t(\textbf{p}) }}{ \sum_{l=1}^{n_i} a_{i l}^{t(\textbf{p})} }.\end{equation}
For $k\in\mathbb{N}$, set $r(k) = \sum_{(i,j)\in D} \lceil k q_{i j} \rceil$, and let
\begin{equation}
\nonumber\Gamma_k:=\{ \textbf{d}=d_1d_2\cdots d_{r(k)} \in D^{r(k)}\,\,:\,\, |\{1\leq s \leq r(k)\,\,:\,\, d_s=(i,j)\}|=\lceil k q_{i j}\rceil \},\end{equation}
i.e. the set of all strings of length $r(k)$ over the alphabet $D$, for which the number of occurrences of the letter $(i,j)$ is equal to $\lceil k q_{i j} \rceil$.  A simple combinatorial argument shows that
\begin{equation}
\label{card2}
|\Gamma_k |=\frac{r(k)!}{\prod_{(i,j)\in D} \lceil k q_{i j}\rceil !}.\end{equation}
Consider the iterated function system $\{ S_{d_1}\circ S_{d_2} \circ \cdots \circ S_{d_{r(k)}} \}_{\textbf{d}\in\Gamma_k}$, and let its associated attractor be denoted by $\Lambda_k$.  Clearly, $\Lambda_k\subset \Gamma$, and the linear part  of each map has diagonal entries $\prod_{(i,j)\in D} a_{i j }^{\lceil k q_{i j} \rceil}$, $\prod_{(i,j)\in D} b_{i}^{ \lceil k q_{i j} \rceil} $.   Let $\pi: D\rightarrow \{1,\ldots,m\}$ denote the projection onto the second coordinate, and let
\begin{equation}\nonumber \tilde{\Gamma}_k=\{ \pi(d_1)\pi(d_2)\cdots\pi(d_{r(k)})\,\,:\,\, d_{1}d_{2}\cdots d_{r(k)}\in\Gamma_k\},\end{equation} denote the projections of the strings in $\Gamma_k$ onto their second coordinates.   Then we have
\begin{equation}\label{card3} |\tilde{\Gamma}_k| =\frac{r(k)!}{\prod_{i=1}^{m} \left(\sum_{j=1}^{n_i} \lceil k q_{i j} \rceil \right)!}.\end{equation}  For each $\sigma_1 \sigma_2 \cdots \sigma_{r(k)}\in\tilde{\Gamma_k}$ the number of elements $d_1 d_2 \cdots d_{r(k)}\in\Gamma_k$ such that $\pi(d_i)=\sigma_i$ for $i=1,2,\ldots, r(k)$ is equal to $|\Gamma_k|/|\tilde{\Gamma}_k|.$  Therefore, combining Stirling's formula with (\ref{card2}) and (\ref{card3}), we deduce
\begin{eqnarray}
\nonumber \dim_H \Lambda_k &=& \frac{\log |\tilde{\Gamma_k}|}{-\sum_{(i,j)\in D} \lceil k q_{i j} \rceil \log b_i } + \frac{\log |\Gamma_k| - \log |\tilde{\Gamma_k}|}{-\sum_{(i,j)\in D} \lceil k q_{i j} \rceil \log a_{i j } } \\
\nonumber & = & \frac{\sum_{i=1}^m p_i \log p_i}{\sum_{i=1}^m p_i \log b_i } + \frac{\sum_{i=1}^m p_i\log p_i - \sum_{(i,j)\in D} q_{i j} \log q_{i j} }{- \sum_{(i,j)\in D} q_{i j}\log a_{i j} } + o(1) \\
\nonumber & = & \frac{\sum_{i=1}^m p_i \log p_i}{\sum_{i=1}^m p_i \log b_i } + t(\textbf{p})+ o(1)=\dim_H \Lambda + o(1).\end{eqnarray}
This completes the proof of the first part of the Lemma.

For the second part we will assume the system is of irrational type, and let
\[
\left\{f_i(x)=\begin{pmatrix}a&0\\0&b\end{pmatrix}x+y_i\right\}_{i=1}^N
\]
be the iterated function system from the first part with attractor $\Lambda_k$. If property (\ref{it1}) is satisfied then we can find a map $S_{i,j}$ in the original system with $\frac{\log a_{ij}}{\log b_i}\notin\Q$. By Lemma \ref{rat1} one of $\frac{\log a}{\log b},\frac{\log a+\log a_{ij}}{\log b+\log b_i},\frac{\log a+2\log a_{ij}}{\log b+2\log b_i}$ must be irrational. One of the three systems

\[
\left\{f_l(x)=\begin{pmatrix}a&0\\0&b\end{pmatrix}x+y_i\right\}_{l=1}^N,
\]
\[
\left\{f_l(x)=S_{i,j}\circ\left(\begin{pmatrix}a&0\\0&b\end{pmatrix}x+y_i\right)\right\}_{l=1}^N,
\]
\[
\left\{f_l(x)=S_{i,j}^2\circ\left(\begin{pmatrix}a&0\\0&b\end{pmatrix}x+y_i\right)\right\}_{l=1}^N,
\]
will then be the required system. The difference between the dimension of the attractor of the new system and the dimension of $\Lambda_k$ can be made arbitrarily small by letting $k\rightarrow\infty$.

If property (\ref{it2}) is satisfied then we use Lemma \ref{rat2}. We can assume that for each $i,j$ $\frac{\log a_{ij}}{\log b_i}=c_i^{-1}\in\Q$ and that $\frac{\log a}{\log b}=c^{-1}\in\Q$. By using property (\ref{it2}) and considering $\R$ as a vector space over $\Q$ we can find $(k,l)\in D$ such that $\{\log a,\log a_{kl}\}$ is linearly independent. If
$$\frac{\log a+\log a_{kl}}{\log b+\log b_k}=\frac{\log a+\log a_{kl}}{c\log a+c_k\log a_kl}\in\Q$$
then by Lemma \ref{rat2} $c=c_k$ and so $\frac{\log a}{\log b}=\frac{\log a_{kl}}{\log b_k}$ for any $(k,l)\in D$ where $\{\log a_{kl},\log a\}$ is linearly independent. If all $\log a_{ij}$ are linearly independent with $\log a$ then this violates the second part of condition (\ref{it2}) and we can deduce
that
$$\frac{\log a+\log a_{ij}}{\log b+\log b_i}\notin\Q$$
for some $(i,j)\in D$ and we can take our iterated function system as
\[
\left\{f_l(x)=S_{i,j}\circ\left(\begin{pmatrix}a&0\\0&b\end{pmatrix}x+y_i\right)\right\}_{l=1}^N.
\]
Finally if there also exists $(r,s)\in D$ where $\log a_{rs}=q\log a$ for some $q\in\Q$ then if
$$\frac{\log a+\log a_{kl}+q\log a}{c\log a+c_k\log b_{kl}+qc_r\log b}\in\Q$$
we can deduce that $c_r=c$ since from the previous part we can assume that $c_k=c$. For this to be true for all such $(r,s)\in D$ would violate condition (\ref{it2}).
\end{proof}
Using this Lemma and then applying \ref{keyprop} completes the proof of Theorem \ref{thm1}. Note that while it is possible that in the Bara\'{n}ski case we could have $a>b$ when we want to apply Proposition \ref{keyprop}, this is easily dealt with by swapping the $x$ and $y$ coordinates.

\ack{We thank Mark Pollicott for useful discussions}.

\bibliographystyle{abbrv}
\bibliography{projections}
\end{document}